\documentclass[12pt]{amsart}
\usepackage[all]{xy}
\usepackage[parfill]{parskip}

\usepackage{verbatim}
\usepackage{color}

\usepackage{amsthm,amsfonts,amssymb,latexsym,bm}
\usepackage{amsmath, amscd, graphicx, latexsym, hyperref, times}
\usepackage{graphicx}
\usepackage[abs]{overpic}
\usepackage{tikz}
\usepackage{tikz-cd}
\usetikzlibrary{arrows, patterns}

\oddsidemargin .25in \theoremstyle{plain}
\theoremstyle{plain}      
\newtheorem{theorem}{Theorem}[section]

\newtheorem{proposition}[theorem]{Proposition}
\newtheorem{lemma}[theorem]{Lemma}

\newtheorem{conjecture}[theorem]{Conjecture}

\theoremstyle{definition} 
\newtheorem{example}[theorem]{Example}

\theoremstyle{remark}     
\newtheorem{remark}[theorem]{Remark}

\allowdisplaybreaks[3]

\newcommand{\dd}{\mathop{}\!\mathrm{d}}
\usepackage{cleveref}
\usepackage{etoolbox}
\crefname{theorem}{Theorem}{Theorems}

\crefname{lemma}{Lemma}{Lemmas}
\AtBeginEnvironment{lemma}{\crefalias{theorem}{lemma}}

\crefname{proposition}{Proposition}{Propositions}
\AtBeginEnvironment{proposition}{\crefalias{theorem}{proposition}}

\crefname{corollary}{Corollary}{Corollaries}
\AtBeginEnvironment{corollary}{\crefalias{theorem}{corollary}}  

\crefname{remark}{Remark}{Remarks}
\AtBeginEnvironment{remark}{\crefalias{theorem}{remark}}

\crefname{conjecture}{Conjecture}{Conjectures}
\AtBeginEnvironment{conjecture}{\crefalias{theorem}{conjecture}}

\crefname{example}{Example}{Examples}
\AtBeginEnvironment{example}{\crefalias{theorem}{example}}

\crefname{definition}{Definition}{Definitions}
\AtBeginEnvironment{definition}{\crefalias{theorem}{definition}}

\crefname{equation}{}{}
\crefname{section}{Section}{Section}
\crefname{figure}{Figure}{Figure}
\usepackage{appendix}
\begin{document}

\title[Sharp rigidity  for quasilinear Liouville equation on manifolds]
{Sharp rigidity for quasilinear Liouville equation on manifolds with nonnegative Ricci curvature}
\author{Xiaohan Cai}
\address{School of Mathematical Sciences, Shanghai Jiao Tong University, Shanghai 200240, China}
\email{xiaohancai@sjtu.edu.cn}

\date{}
\begin{abstract}
We study the quasilinear Liouville equation
\[
-\Delta_n u=e^u
\]
on complete noncompact Riemannian manifolds  with nonnegative Ricci curvature.
Our first  result shows that, if a solution $u$ satisfies the optimal logarithmic lower bound
\[
u(x)\ge -\frac{n^2}{n-1}\log r(x)+o(\log r(x)) \quad \text{as }r(x)\to+\infty,
\]
then the underlying manifold is isometric to the Euclidean space and $u$ is a standard bubble solution. Both the leading coefficient and the remainder term in the assumption are sharp.
The key ingredient in the proof is the connection between the logarithmic lower bound and a sharp upper bound on the total volume of the solution. 

We also formulate a conjecture concerning the interaction between the sub-logarithmic decay of solutions and the underlying geometry, and prove it for $n=2$, as well as for $n\ge 3$ under a strengthened assumption.
\end{abstract}
\maketitle
\section{Introduction}
	The Liouville equation in dimension two
 \begin{align}\label{eq. Liouville}
     -\Delta u=e^{u}
 \end{align}
  is a basic equation that has remained an active research subject for over a century.
 Geometrically, a solution of \cref{eq. Liouville} on $\mathbb{R}^2$ gives rise to a conformal metric $g=e^{u}g_0$ with constant Gaussian curvature $K_g\equiv \frac{1}{2}$.
 Physically, it  also arises in plasticity, nonlinear elasticity, and membrane buckling.
 The classification of entire solutions to it is therefore of fundamental importance.

 In the Euclidean setting, Liouville \cite{Lio1853} obtained an elegant representation formula of all entire solutions 
  in terms of meromorphic functions.
 A landmark classification was later established by Chen and Li \cite{CL91} under the \emph{finite-volume} assumption
 \begin{align*}
     \int_{\mathbb{R}^2}e^{u}<+\infty.
 \end{align*}
 Geometrically,  their result states that
  any finite-volume solution to \cref{eq. Liouville} is the pullback of the round metric on $\mathbb{S}^2$ via stereographic projection.
  This classification result  inspired extensive subsequent work in the field of conformal geometry and PDE, see, e.g., \cite{CL97, Lin98, CK00, CW01, CLO06}.
 Several alternative proofs were subsequently given in \cite{CW94, CK95, JW02,  BLdS04, HW06, LT20}. 
 Classification results  have also been established  under other assumptions, including stability \cite{Far07}, boundedness \cite{EGLX25}, and concavity \cite{BEL23}.
 We refer the reader to the survey \cite{CL24survey} for further details.

 A higher dimensional analogue of the Liouville equation \cref{eq. Liouville} is the $n$-Liouville equation
 \begin{align}\label{eq. n Liouville equation}
     -\Delta_n u=e^{u},
 \end{align}
 where 
 \begin{align*}
     \Delta_n u:=div(|\nabla u|^{n-2}\nabla u)
 \end{align*}
 is a conformally invariant analogue of the Laplace operator in $n$-dimensional space.
 
 Under the \emph{finite-volume} assumption $\int_{\mathbb{R}^n}e^{u}<+\infty,$
 Esposito \cite[Theorem 1.1]{Esp18} classified the solutions to \cref{eq. n Liouville equation}. More explicitly,
 \begin{theorem}[\cite{Esp18}]\label{thm. Esposito 18}
     Let $u\in C^3(\mathbb{R}^n)$ be a solution to the equation
     \begin{align*}
     \begin{cases}
         -\Delta_n u=e^u &\text{in }\mathbb{R}^n,\\
         \int_{\mathbb{R}^n}e^u<+\infty,
     \end{cases}
     \end{align*}
     then the solution is given by bubble solutions, namely
     \begin{align}\label{eq. bubble solution}
         u(x)=\log \frac{c_n\lambda^n}
         {\left(1+\lambda^{\frac{n}{n-1}}|x-x_0|^{\frac{n}{n-1}}\right)^n}
     \end{align}
     for some $\lambda>0$, $x_0\in\mathbb{R}^n$ and $c_n:=n(\frac{n^2}{n-1})^{n-1}$.
 \end{theorem}

 Esposito's proof relies on a sophisticated asymptotic analysis, together with the Pohozaev identity and the isoperimetric inequality.
 Among other results,
 Ciraolo, Esposito and Li  \cite[Theorem 1.2]{CEL26}  presented  an alternative proof using the P-function method.

 Unlike the Euclidean case, where the classification is complete, on manifolds with nonnegative Ricci curvature both flexibility and rigidity phenomena exist. On the one hand, there are abundant manifolds with nonnegative Ricci curvature admitting solutions with exotic asymptotic behavior;
 on the other hand, this setting retains enough comparison tools and geometric inequalities to force rigidity of the  ambient manifold under suitable asymptotic control of the solution. These competing features make the classification problem both subtle and meaningful.


 In this paper, we investigate the interplay between solutions to the Liouville equation \cref{eq. n Liouville equation} and the geometry of the underlying manifold, exploring both  flexibility and rigidity phenomena.

 Our first main result is a sharp classification result for solutions to the  equation \cref{eq. n Liouville equation} on $n$-dimensional complete manifolds with nonnegative Ricci curvature for all $n\geq 2$.
 \begin{theorem}\label{thm. sharp classification}
     Let $(M^n,g)$, $n\geq 2$, be a complete noncompact manifold with nonnegative Ricci curvature. Assume $u\in C^3(M^n)$ is a solution to
     \begin{align}\label{eq. n Liouville equation on M}
         -\Delta_n u=e^u\quad 
         \text{on }M^n.
     \end{align}
     Assume in addition that
     \begin{align}\label{eq. sharp logarithmic lower bound}
         u(x)\geq -\frac{n^2}{n-1}\log r(x)
         +o(\log r(x)),\quad \forall r\gg 1,
     \end{align}
     where $r$ is the distance function on $(M^n,g)$ to some fixed point $p$. Then $(M^n,g)$ is isometric to the flat Euclidean space $\mathbb{R}^n$ and $u$ is given by \cref{eq. bubble solution}.

     Moreover, the coefficient $-\frac{n^2}{n-1}$ is optimal, in the sense that, for any $\gamma>\frac{n^2}{n-1}$, there exists a non-flat, complete manifold with nonnegative Ricci curvature, admitting a finite-volume solution to  \cref{eq. n Liouville equation on M} satisfying $u(x)\sim -\gamma\log r(x)$ for $r\gg 1$.
 \end{theorem}
   
 The two-dimensional case has been progressively understood.
 Catino and Monticelli \cite[Theorem 1.6]{CM26} first proved this rigidity phenomenon under the assumption
 \begin{align}\label{eq. CM assumption_two dim}
         \int_M e^u<+\infty,\quad \text{and }
         u(x)\geq -4\log r(x)-2\gamma \log\log r(x),\quad \forall r\gg 1,
 \end{align}
 for some $\gamma\in[0,1)$.  
 This  was later improved by the author and Lai \cite[Theorem 2]{CL24} under a relaxed assumption
 \begin{align}\label{assumption of CL24}
         \int_M e^u<+\infty,\quad \text{and }
         u(x)\geq -4\log r(x)+o(\log r(x)),\quad \forall r\gg 1,
 \end{align}
 In \cite[Theorem 1.2]{CFP25}, Ciraolo, Farina and Polvara further made a breakthrough. By exploiting the P-function method, they obtained the classification result only under an assumption on the lower bound of $u$:
 \begin{align}\label{assumption of CFP}
     u(x)\geq -4\log r(x)-\log F(r(x)),\quad 
     \forall r\gg 1,
 \end{align}
 for some  positive, nondecreasing function $F$ satisfying \footnote{A particular choice of $F$ is $F(t)=\log t$.}
 \begin{align*}
     \int_2^{+\infty}\frac{1}{tF(t)}dt=+\infty.
 \end{align*}
 Later, Ou \cite[Theorem 1.1]{Ou26} slightly relaxed \cref{assumption of CFP}. Very recently, Ciraolo, Farina and Gatti \cite[Corollary 1.1.2]{CFG26} further improved it, allowing for  a power of $F$ as the remainder term:
 \begin{align*}
     u(x)\geq -4\log r(x)-F^\beta(r(x)),\quad 
     \forall r\gg 1,
 \end{align*}
 for some constant $\beta\in(0,1)$.

 For the $n$-dimensional case, 
Sun and Wang \cite[Theorem II]{SW25}, adapting the P-function method initiated by \cite{CFP25}, proved the classification result under the $n$-dimensional analogue assumption of \cref{assumption of CFP}: 
\begin{align*}
    u(x)\geq -\frac{n^2}{n-1}\log r(x)
         -\frac{n}{2(n-1)}\log F(r(x)),\quad \forall r\gg 1,
\end{align*}
Recently, Catino, Monticelli and Roncoroni \cite[Theorem 1.1]{CMR26} further relaxed the remainder term.
They  \cite[Theorem 1.3]{CMR26} also obtained the classification under the $n$-dimensional analogue assumption of \cref{assumption of CL24}:
\begin{align*}
         \int_{M^n} e^u<+\infty,\quad \text{and }
         u(x)\geq -\frac{n^2}{n-1}\log r(x)+o(\log r(x)),\quad \forall r\gg 1,
 \end{align*}

Compared with previous work, the novelty of \cref{thm. sharp classification} is twofold. First, it does not require the integrability of $e^u$. 
Second,  the leading coefficient $-\frac{n^2}{n-1}$ is sharp  and the remainder term $o(\log r(x))$ is the natural borderline for the classification. This constitutes the sharp analogue of \cref{thm. Esposito 18} in the setting of manifolds with nonnegative Ricci curvature.
Even in the two-dimensional case, this result is new, as it removes the finite-volume assumption in \cite{CL24}.

 The proof of \cref{thm. sharp classification} relies on the following sharp volume estimate.
 \begin{proposition}\label{prop. finite volume}
     Let $(M^n,g), n\geq 2$, be a complete noncompact manifold with nonnegative Ricci curvature. Assume $u\in C^3(M^n)$ is a solution to \cref{eq. n Liouville equation on M} satisfying
    \begin{align}\label{eq. logarithmic decay}
        u(x)\geq
        -\gamma\log r(x)
        +o(\log r(x))\quad 
        \text{as }r(x)\to+\infty,
    \end{align}
    for some constant $\gamma>0$.
    Then there is an upper bound of the volume of the solution:
    \begin{align*}
        \int_M e^u
        \leq |\mathbb{S}^{n-1}|\mathrm{AVR}(M^n,g)\, \gamma^{n-1}<+\infty,
    \end{align*}
    where $\mathrm{AVR}(M^n,g)$ is the asymptotic volume ratio defined by
    \begin{align}\label{def. AVR}
        \mathrm{AVR}(M^n,g):=\lim_{r\to+\infty}
        \frac{Vol(B(p,r))}{|\mathbb{B}^n|r^n}
        \in[0,1].
    \end{align}
 \end{proposition}
  Applying the sharp isoperimetric inequality of Brendle \cite{Bre23} then yields equality in the volume estimates, which forces the manifold to be Euclidean.

  Except for the role played in the proof of \cref{thm. sharp classification}, \cref{prop. finite volume} itself  is of independent interest. It provides a sharp criterion for integrability of the source term in terms of the logarithmic decay of the solution (see also \cref{lem. upper bound of L1 norm}),
  is reminiscent of the classical logarithmic potential estimates in the study of conformally invariant equations on Euclidean spaces, for example, in \cite{CL97, CQY00, MQ21, Li24},
  and extends the link between these two concepts to the setting of manifolds.

  In the second part of this paper, we further study how  this connection interacts with the geometry of the underlying manifold.
  We formulate a conjecture as the converse of \cref{prop. finite volume}, which addresses the geometric consequences of a solution with \emph{non-logarithmic} behavior (see \cref{conj. converse of prop}). 

  With several different arguments, we prove that \cref{conj. converse of prop} is true for $n=2$ (see \cref{thm. converse for two dimensional}), and is true for $n\geq 3$ under a strengthened assumption (see \cref{thm. converse for higher dimensional}). The proof of these results relies on a key technical lemma (\cref{lem. lower bound of L1 norm}), which is also of independent interest.

This paper is organized as follows. In \cref{sec. proof}, we first establish \cref{prop. finite volume} and then complete the proof of \cref{thm. sharp classification}. 
In \cref{sec. discussion}, we formulate the converse of \cref{prop. finite volume} as a conjecture. Then we construct examples to demonstrate the sharpness of \cref{prop. finite volume} and to illustrate the necessity of the assumptions in the conjecture.
Finally, we prove \cref{thm. converse for two dimensional} and \cref{thm. converse for higher dimensional}.

\section{Proof of \cref{thm. sharp classification}}\label{sec. proof}

This section is devoted to proving \cref{thm. sharp classification}. The strategy is as follows.
 On the one hand, \cref{prop. finite volume} shows that a logarithmic decay of the solution implies a sharp upper bound of the total volume.
 On the other hand, the sharp isoperimetric inequality (see \cite{Bre23}) provides a corresponding lower bound of the volume of the solution. Therefore, under the Euclidean sharp  bound \cref{eq. sharp logarithmic lower bound},
 these two estimates force equality in the volume bounds, and the rigidity case of the sharp isoperimetric inequality forces the manifold to be Euclidean.

We  first prove \cref{prop. finite volume}. 
It relies on \cite[Lemma 4.2]{CMR26}:
\begin{lemma}[\cite{CMR26}]\label{lem. upper bound of total curvature}
     Let $\Omega\subset (M^n,g)$ be a smooth compact domain and let $K\Subset \Omega$ be a smooth compact domain. Assume $u\in W^{1,n}(\Omega)\cap L^{\infty}(\Omega)$ satisfies
    \begin{align*}
        -\Delta_nu=f\geq 0
    \end{align*}
    for some nonnegative function $f\in L^1(\Omega)$. Then
    \begin{align}\label{ineq. lower bound of n-subharmonic function}
        \int_K f
        \leq Cap_n(K,\Omega)
        \left(-\inf_{\partial\Omega} u
        +\sup_{ K}u\right)^{n-1},
    \end{align}
    where
    \begin{align*}
        Cap_n(K,\Omega)
        :=\inf\left\{
        \int_{\Omega}|\nabla \varphi|^n:\varphi\in W_0^{1,n}(\Omega), \varphi\geq 1 \text{ on } K.
        \right\}.
    \end{align*}
\end{lemma}
\begin{proof}[Proof of \cref{prop. finite volume}]
    Fix $0<R_1<R_2$, by \cref{lem. upper bound of total curvature}, we get
    \begin{align}\label{eq. upper bound of total curvature on R1 ball}
        \int_{B(p,R_1)}e^u
        \leq Cap_n(B(p,R_1), B(p,R_2))
        \left(-\inf_{\partial B(p,R_2)} u
        +\sup_{ B(p,R_1)}u\right)^{n-1}.
    \end{align}
    
    In the following, we estimate the two factors of the right hand side of \cref{eq. upper bound of total curvature on R1 ball}.
    Consider the test function
    \begin{align}\label{eq. test function varphi}
        \varphi(x):=\frac{\int_{r(x)}^{R_2}A^{-\frac{1}{n-1}}(t)dt}
        {\int_{R_1}^{R_2}A^{-\frac{1}{n-1}}(t)dt}, \quad 
        x\in B(p,R_2)\setminus B(p,R_1),
    \end{align}
    where 
    \begin{align*}
        A(t):=Area (\partial B(p,t)).
    \end{align*}
    It is straightforward to see that $\varphi|_{\partial B(p,R_1)}=1,\, \varphi|_{\partial B(p,R_2)}=0$ and
    \begin{align*}
        \nabla \varphi(x)
        =\frac{-A^{-\frac{1}{n-1}}(r(x))}{\int_{R_1}^{R_2} A^{-\frac{1}{n-1}}(t)d t}
        \nabla r,\quad
        x\in B(p,R_2)\setminus B(p,R_1).
    \end{align*}
    Therefore we could derive an upper bound of $Cap_n(B(p,R_1), B(p,R_2))$:
    \begin{align}
        Cap_n(B(p,R_1), B(p,R_2))
        \leq& \int_{B(p,R_2)\setminus B(p,R_1)}
        |\nabla \varphi|^n\notag\\
        =&\int_{R_1}^{R_2}
        \left(
        \int_{\partial B(p,r)}\frac{A^{-\frac{n}{n-1}}(r)}
        {\left(\int_{R_1}^{R_2} A^{-\frac{1}{n-1}}(t) d t\right)^n} 
        \right) d r\notag\\
        =&\int_{R_1}^{R_2}
        \frac{A^{-\frac{1}{n-1}}(r)}
        {\left(\int_{R_1}^{R_2} A^{-\frac{1}{n-1}}(t)d t\right)^n} 
        d r\notag\\
        =&
        \left(
        \int_{R_1}^{R_2} A^{-\frac{1}{n-1}}(t)\dd t
        \right)^{-(n-1)}.\label{ineq. upper bound of capacity}
    \end{align}

    By the definition of the asymptotic volume ratio $\mathrm{AVR}(M^n,g)$ and the Bishop-Gromov volume comparison theorem, for any $\epsilon>0$, there exists $R_\epsilon>0$ such that
    \begin{align*}
        \mathrm{AVR}(M^n,g)+\epsilon\geq \frac{A(r)}{|\mathbb{S}^{n-1}|r^{n-1}},\quad \forall r\geq R_\epsilon.
    \end{align*}
    Equivalently,
    \begin{align*}
        |\mathbb{S}^{n-1}|^{\frac{1}{n-1}} 
        A(r)^{-\frac{1}{n-1}}
        \geq (\mathrm{AVR}(M^n,g)+\epsilon)^{-\frac{1}{n-1}}\frac{1}{r},\quad \forall r\geq R_\epsilon.
    \end{align*}
    Integrating this inequality from $R_1$ to $R_2$ yields that, for any $R_2>R_1\geq R_\epsilon$,
    \begin{align*}
        |\mathbb{S}^{n-1}|
        \left(\int_{R_1}^{R_2} A^{-\frac{1}{n-1}}(t)d t\right)^{n-1}
        \geq (\mathrm{AVR}(M^n,g)+\epsilon)^{-1}
        (\log R_2-\log R_1)^{n-1}.
    \end{align*}
    Combining this with \cref{ineq. upper bound of capacity}, we obtain that, for any $R_2>R_1\geq R_\epsilon$,
    \begin{align}\label{eq. ingredient 1}
        Cap_n(B(p,R_1), B(p,R_2))
        \leq 
        |\mathbb{S}^{n-1}|(\mathrm{AVR}(M^n,g)+\epsilon)
        \left(\log R_2-\log R_1\right)^{-(n-1)}.
    \end{align}
    On the other hand, the assumption \cref{eq. logarithmic decay} implies
    \begin{align}\label{eq. ingredient 2}
        -\inf_{\partial B(p,R_2)} u
        +\sup_{ B(p,R_1)}u
        \leq \gamma\log R_2+o(\log R_2)
        +\sup_{ B(p,R_1)}u,\quad 
        \forall R_2\gg 1.
    \end{align}
    
    For any $\epsilon>0$ and $R_2> R_1\geq R_\epsilon$,  inserting \cref{eq. ingredient 1} and \cref{eq. ingredient 2} into \cref{eq. upper bound of total curvature on R1 ball}, we derive
    \begin{align*}
        \int_{B(p,R_1)}e^u
        \leq |\mathbb{S}^{n-1}|
        (\mathrm{AVR}(M^n,g)+\epsilon)
        \left(\frac{\gamma\log R_2+o(\log R_2)+\sup_{B(p,R_1)}u}{\log R_2-\log R_1}\right)^{n-1}.
    \end{align*}
    Letting $R_2$ go to $\infty$, we obtain, for any $\epsilon>0$ and $R_1\geq R_\epsilon$,
    \begin{align*}
    \int_{B(p,R_1)}e^u
    \leq |\mathbb{S}^{n-1}|(\mathrm{AVR}(M^n,g)+\epsilon)
    \gamma^{n-1}.
    \end{align*}
    Finally, letting $R_1$ go to $\infty$ and then letting $\epsilon$ go to zero, we achieve the desired estimate
    \begin{align*}
        \int_M e^u\leq |\mathbb{S}^{n-1}|
        \mathrm{AVR}(M^n,g)\,\gamma^{n-1}<+\infty.
    \end{align*}
\end{proof}
    Inspecting the proof of \cref{prop. finite volume}, we can actually prove the following slightly more general result.
    \begin{lemma}\label{lem. upper bound of L1 norm}
        Let $(M^n,g), n\geq 2$, be a complete noncompact manifold with nonnegative Ricci curvature. Assume $u\in C^3(M^n)$ is a solution to
    \begin{align*}
        -\Delta_n u=f,\quad
        \text{on }M^n,
    \end{align*}
    for some nonnegative function $f\in L^{\infty}_{\text{loc}}$. Assume 
    \begin{align*}
        \liminf_{r(x)\to+\infty}
        \frac{u(x)}{\log r(x)}>-\infty,
    \end{align*}
    Then $f\in L^1(M)$. More precisely,
    \begin{align*}
        \int_M f
        \leq |\mathbb{S}^{n-1}|\mathrm{AVR}(M^n,g)\, \left(
        -\liminf_{r(x)\to+\infty}
        \frac{u(x)}{\log r(x)}
        \right)^{n-1}<+\infty,
    \end{align*}
    where $\mathrm{AVR}(M^n,g)$ is defined in \cref{def. AVR}.
    \end{lemma}
    \begin{remark}
        The nonnegativity of $f$ in \cref{lem. upper bound of L1 norm} could be removed (see \cref{lem. upper bound of L1 norm_second}).
    \end{remark}

Now we are ready to prove \cref{thm. sharp classification}. The arguments have  appeared in \cite[Proposition 2.2]{CL24}, \cite[Lemma 4.1]{CMR26} and could be traced back to \cite[Lemma 1.1]{CL91}. For the courtesy of readers, we present a detailed proof.

\begin{proof}[Proof of \cref{thm. sharp classification}]
    By \cref{prop. finite volume} and  \cref{eq. sharp logarithmic lower bound}, we know that
    \begin{align}\label{eq. upper bound of volume}
        \int_M e^u\leq 
        |\mathbb{S}^{n-1}| \mathrm{AVR}(M^n,g)
        \left(\frac{n^2}{n-1}\right)^{n-1}.
    \end{align}

    On the other hand, consider 
    \begin{align*}
        \Omega_t:=\{x\in M: u(x)>t\},\quad 
        F(t):=\int_{\Omega_t}e^u.
    \end{align*}
    Then the integrability of $e^u$ over $M$ implies that $Vol(\Omega_t)<+\infty$ for any $t\in\mathbb{R}$.
    It follows from \cref{eq. n Liouville equation on M} and divergence theorem that
    \begin{align*}
        F(t)=-\int_{\Omega_t}\Delta_n u
        =\int_{\partial\Omega_t}|\nabla u|^{n-1}.
    \end{align*}
    By the coarea formula, there holds
    \begin{align*}
        F'(t)=-\int_{\partial\Omega_t}
        \frac{e^u}{|\nabla u|}
        =-e^t\int_{\partial\Omega_t}\frac{1}{|\nabla u|}.
    \end{align*}
    Then the Holder's inequality and the sharp isoperimetric inequality due to \cite[Corollary 1.3]{Bre23} (see also \cite[Theorem 1.1]{BK23} and \cite[Theorem 1.1]{Kri24}), we have
    \begin{align*}
        -\left(F^{\frac{n}{n-1}}\right)'(t)
        =&-\frac{n}{n-1}F^{\frac{1}{n-1}}(t)
        F'(t)
        =\frac{n}{n-1}e^t
        \left(\int_{\partial\Omega_t}|\nabla u|^{n-1}\right)^{\frac{1}{n-1}}
        \int_{\partial\Omega_t}\frac{1}{|\nabla u|}\\
        \geq &\frac{n}{n-1}e^t 
        \text{Area}(\partial\Omega_t)^{\frac{n}{n-1}}\\
        \geq& \frac{n^2}{n-1}
        \left(|\mathbb{S}^{n-1}|\mathrm{AVR}(M^n,g)\right)^{\frac{1}{n-1}}e^t
        Vol(\Omega_t).
    \end{align*}
    Integrating this inequality from $-\infty$ to $+\infty$ yields
    \begin{align*}
        \left(\int_M e^u\right)^{\frac{n}{n-1}}
        \geq& \frac{n^2}{n-1}
        \left(|\mathbb{S}^{n-1}|\mathrm{AVR}(M^n,g)\right)^{\frac{1}{n-1}}
        e^t Vol(\Omega_t)
        \int_{-\infty}^{+\infty}e^t Vol(\Omega_t)dt\\
        =&\frac{n^2}{n-1}
        \left(|\mathbb{S}^{n-1}|\mathrm{AVR}(M^n,g)\right)^{\frac{1}{n-1}}
        \int_M e^u.
    \end{align*}
    Rearranging it, we get
    \begin{align}\label{eq. lower bound of volume}
        \int_M e^u \geq
        |\mathbb{S}^{n-1}| \mathrm{AVR}(M^n,g)
        \left(\frac{n^2}{n-1}\right)^{n-1}.
    \end{align}
    Then \cref{eq. lower bound of volume} and \cref{eq. upper bound of volume} force an equality. Inspecting the proof of \cref{eq. lower bound of volume} shows that
    \begin{align*}
        \text{Area}(\partial\Omega_t)^{\frac{n}{n-1}}
        = \left(|\mathbb{S}^{n-1}|\mathrm{AVR}(M^n,g)\right)^{\frac{1}{n-1}}
        Vol(\Omega_t),\quad 
        \forall t\in\mathbb{R}.
    \end{align*}
    Therefore, the rigidity part of the sharp isoperimetric inequality  forces $(M^n,g)$ to be isometric to the flat Euclidean space $\mathbb{R}^n$ (see  \cite[Theorem 1.2]{Bre23}). Finally, the classification of $u$ follows from \cref{thm. Esposito 18}.

    For the sharpness of the coefficient $-\frac{n^2}{n-1}$ in \cref{eq. sharp logarithmic lower bound}, readers are referred to \cite[Example 3.4]{CMR26}.
    For any $\gamma>\frac{n^2}{n-1}$, they constructed a non-flat complete warped product manifold $(M^n_\gamma,g_\gamma)$ with nonnegative Ricci curvature and a function $u_\gamma$ satisfying
    \begin{align*}
        -\Delta_n u_\gamma=e^{u_\gamma}\quad
        \text{in }M^n_\gamma,
    \end{align*}
    and 
    \begin{align}\label{example. sharpness of leading coefficient}
        \lim_{r(x)\to+\infty}\frac{u_\gamma(x)}{\log r(x)}=-\gamma,\quad 
        \mathrm{AVR}(M^n_\gamma,g_\gamma)
        \in(0,1),\quad 
        \int_{M_\gamma}e^{u_\gamma}
        =|\mathbb{S}^{n-1}|
        \mathrm{AVR}(M^n_\gamma,g_\gamma)
        \gamma^{n-1}.
    \end{align}
    This completes the proof of \cref{thm. sharp classification}.
\end{proof}
\begin{remark}
    By virtue of \cref{lem. upper bound of L1 norm}, the assumption \cref{eq. sharp logarithmic lower bound} in \cref{thm. sharp classification} can be relaxed to  the following slightly weaker condition:
    \begin{align*}
        \liminf_{r(x)\to+\infty}
        \frac{u(x)}{\log r(x)}\geq -\frac{n^2}{n-1}.
    \end{align*}
\end{remark}

 \section{Converse of \cref{prop. finite volume}}\label{sec. discussion}
The proof of \cref{thm. sharp classification} reveals a subtle interplay between the asymptotic behavior of the solution to the Liouville equation \cref{eq. n Liouville equation on M}, the $L^1$ norm of the source term of the equation, and the asymptotic volume ratio of the underlying manifold.
 
 In this section, we further explore this analytic-geometric interaction.
 The examples \cref{example. sharpness of leading coefficient} show that
 the estimate in \cref{prop. finite volume} is sharp. It is therefore natural to ask whether the converse of \cref{prop. finite volume} also holds.
 
 To formulate such a converse, observe that merely negating the logarithmic decay condition \cref{eq. logarithmic decay} is not enough to rule out finiteness of $\int_M e^u$.
 In fact, 
 \cite[Example 4.7]{CMR26} constructs   complete manifolds with nonnegative Ricci curvature and  solutions to \cref{eq. n Liouville equation on M} satisfying
 \begin{align}\label{example vanishing AVR}
     \lim_{r(x)\to+\infty}\frac{u(x)}{\log r(x)}=-\infty,\quad 
     \mathrm{AVR}(M^n,g)=0,\quad 
     \int_M e^u<+\infty,
 \end{align}
 
On the other hand, combining \cref{prop. finite volume} with the lower bound \cref{eq. lower bound of volume} yields the formal inequality
\begin{align*}
    |\mathbb{S}^{n-1}|
    \left(\frac{n^2}{n-1}\right)^{n-1}
    \leq \frac{\int_M e^u}{\mathrm{AVR}(M^n,g)}
    \leq |\mathbb{S}^{n-1}|
    \left(-\liminf_{r(x)\to+\infty}
    \frac{u(x)}{\log r(x)}
    \right)^{n-1}.
\end{align*}
 This motivates the following conjecture, which provides a more complete description of the interplay between the asymptotic behavior of solutions to \cref{eq. n Liouville equation on M} and the asymptotic geometry of the underlying manifold.
 \begin{conjecture}\label{conj. converse of prop}
      Let $(M^n,g), n\geq 2$, be a complete noncompact manifold with nonnegative Ricci curvature. Assume $u\in C^3(M^n)$ is a solution to \cref{eq. n Liouville equation on M}
    satisfying
    \begin{align}\label{eq. sub-logarithmic decay}
        \liminf_{r(x)\to+\infty}
        \frac{u(x)}{\log r(x)}
        =-\infty,
    \end{align}
    then
    \begin{align*}
        \frac{\int_{M^n}e^u}{\mathrm{AVR}(M^n,g)}=+\infty
    \end{align*}
    in the sense that
    \begin{align}\label{two cases in conj}
        \text{either} \int_M e^u<+\infty \text{ and } 
        \mathrm{AVR}(M^n,g)=0,\quad
        \text{or} \int_M e^u=+\infty,
    \end{align}
    where $\mathrm{AVR}(M^n,g)$ is defined in \cref{def. AVR}.
 \end{conjecture}

Both cases in \cref{two cases in conj} can indeed occur.
The example \cref{example vanishing AVR}  falls into the first case. For the second one, \cite[Example 4.6]{CMR26} constructs a one-dimensional solution $u$ to \cref{eq. n Liouville equation on M} on $\mathbb{R}^n$ satisfying 
\begin{align*}
    \int_{\mathbb{R}^n}e^u=+\infty,\quad 
    \text{and }
    u(x_1,\dots,x_n)\sim -a|x_1| \quad 
    \text{as }|x_1|\to +\infty.
\end{align*}
for some constant $a>0$. We note that this solution is uniformly bounded from above on $\mathbb{R}^n$. The classification of such solutions to the Liouville equation \cref{eq. Liouville} on $\mathbb{R}^2$ was achieved by \cite[Theorem 1.6]{EGLX25} via complex-analytic methods, and alternatively by \cite{CLZ26}.

With this example in mind, it is natural to ask whether the second case in \cref{two cases in conj} can be sharpened to $\int_M e^u=+\infty$ and $\mathrm{AVR}(M^n,g)>0$.
However,  the following examples show that such a sharpening can not be true in general.
\begin{example}\label{example mine}
    For $n\geq 2$, we shall construct complete manifold $(M^n,g)$ with nonnegative Ricci curvature, admitting a function $u$ satisfying the equation \cref{eq. n Liouville equation on M} and
    \begin{align*}
        \liminf_{r(x)\to+\infty}\frac{u(x)}{\log r(x)}=-\infty,\quad 
        \int_M e^u=+\infty,\quad 
        \text{and }
        \mathrm{AVR}(M^n,g)=0.
    \end{align*}
    For $n=2$, we regard the flat cylinder as a conformal metric on $\mathbb{R}^2\setminus\{0\}$:
    \begin{align*}
        \left(\mathbb{S}^1\times \mathbb{R}^1, g\right)
        =\left(\mathbb{R}^2\setminus\{0\},\frac{1}{x_1^2+x_2^2}g_0
        \right),
    \end{align*}
    where $g_0$ is the Euclidean metric on $\mathbb{R}^2$ and $x_1,x_2$ are the Cartesian coordinates on it. We consider the family of functions
    \begin{align}\label{solution on cylinder with infinite volume}
        u_t(x_1,x_2):=
        2\log \frac{2e^{x_1}}{1+t^2+2te^{x_1}\cos x_2+e^{2x_1}}
        +\log ({x_1^2+x_2^2})+\log 2,
    \end{align}
    where $t\geq 0$ is a constant. To verify that $u_t$ satisfies the desired properties, let $w_t(x_1,x_2):=u_t(x_1,x_2)
    -\log  ({x_1^2+x_2^2})$. Then it's direct to check that (see also \cite[Theorem 1.6]{EGLX25})
    \begin{align*}
        -\Delta_{g_0} w_t=e^{w_t} \quad 
        \text{in }\mathbb{R}^2,\quad 
        \text{and }\int_{\mathbb{R}^2}e^{w_t}=+\infty.
    \end{align*}
    Therefore, 
    \begin{align*}
        -\Delta_g u_t=-(x_1^2+x_2^2)\Delta_{g_0}(w_t+\log (x_1^2+x_2^2))
        =e^{\log (x_1^2+x_2^2)}e^{w_t}
        =e^{u_t},\\
        \int_{\mathbb{S}^1\times \mathbb{R}^1}e^{u_t}\,\text{d}vol_g
        =\int_{\mathbb{R}^2}
        e^{w_t+\log(x_1^2+x_2^2)}
        \frac{1}{x_1^2+x_2^2}
        \,\text{d}x_1\text{d}x_2
        =\int_{\mathbb{R}^2}e^{w_t}=+\infty,\\
        \liminf_{r(x)\to+\infty}
        \frac{u_t(x)}{\log r(x)}
        =\liminf_{x_1^2+x_2^2\to+\infty}
        \frac{u_t(x_1,x_2)}{\log\log\sqrt{x_1^2+x_2^2}}
        \leq \liminf_{x_1^2\to+\infty}
        \frac{u_t(x_1,0)}{\log \log |x_1|}
        =-\infty.
    \end{align*}

    For $n\geq 3$, we consider the product manifold
    \begin{align*}
        (M^n,g)=
        \left(\mathbb{R}\times\mathbb{R}\times N^{n-2},ds^2+dt^2+g_N\right),
    \end{align*}
    where $(N^{n-2},g_N)$ is a compact manifold with nonnegative Ricci curvature.
    Recall that \cite[Example 4.5]{CMR26} constructed a one-dimensional solution $w(t)$ on $\mathbb{R}$, satisfying
    \begin{align*}
        -(|w'|^{n-2}w')'=e^w\quad \text{on }\mathbb{R},\quad 
        w(t)\sim -a|t|,\quad \text{and }
        \int_{\mathbb{R}}e^{w(t)}\,\text{d}t<+\infty,
    \end{align*}
    for some constant $a>0$. Therefore, the function 
    \begin{align*}
        u(x)=u(s,t,y):=w(t)
    \end{align*}
    is a solution to the equation \cref{eq. n Liouville equation on M} on $\mathbb{R}\times\mathbb{R}\times N$, satisfying
    \begin{align*}
         \int
         _{\mathbb{R}\times\mathbb{R}\times N} e^u
        =\int_N
        \int_{\mathbb{R}}
        \int_{\mathbb{R}}e^{w(t)}
        \dd t\dd s \dd vol_{g_N}
        =+\infty,\quad 
        \liminf_{r(x)\to+\infty}
        \frac{u(x)}{\log r(x)}
        \leq \liminf_{t\to +\infty}
        \frac{w(t)}{\log t}
        =-\infty.       
    \end{align*}
    This completes our examples.
\end{example}

We next present two results addressing \cref{conj. converse of prop}.
First, we prove that the conjecture holds for $n=2$.
\begin{theorem}\label{thm. converse for two dimensional}
    Let $(M^2,g)$, be a complete noncompact surface with nonnegative Gauss curvature. Assume $u\in C^3(M^2)$ is a solution to
    \begin{align*}
        -\Delta u=e^{u},\quad
        \text{on }M^2,
    \end{align*}
    Assume that
    \begin{align*}
        \liminf_{r(x)\to+\infty}
        \frac{u(x)}{\log r(x)}
        =-\infty,
    \end{align*}
    then
    \begin{align*}
        \frac{\int_{M^2}e^u}{\mathrm{AVR}(M^2,g)}=+\infty
    \end{align*}
    in the sense of \cref{two cases in conj}.
\end{theorem}
\begin{proof}
    After passing to the universal covering, we could assume, without loss of generality, that $(M^2,g)$ is orientable. 
    By the classification results of Cohn-Vossen \cite{C-V35} and Huber \cite{Hub57}, $(M^2,g)$ is either isometric to the flat cylinder $\mathbb{S}^1\times\mathbb{R}^1$, or conformal to $(\mathbb{R}^2,g_0)$.

    In the first case, since $\mathrm{AVR}(\mathbb{S}^1\times\mathbb{R}^1)=0$, the conclusion holds trivially.

    In the latter case, we could write $(M^2,g)$ as $(\mathbb{R}^2,e^{2\phi}g_0)$ for some smooth function $\phi$ on $\mathbb{R}^2$. 
    \begin{itemize}
        \item If $\int_M e^u=+\infty$, then the conclusion holds vacuously.
        \item If $\int_M e^u<+\infty$, then \cite[Proposition 2.3]{CL24} have established that $\mathrm{AVR}(M^2,g)=0$. Roughly speaking, they obtained this after deriving
        \begin{align*}
            u(x)\geq -\frac{1}{2\pi}\frac{\int_M e^u}{\mathrm{AVR}(M^2,g)}\log r(x)-C,\quad 
            \forall r(x)\gg 1.
        \end{align*}
    \end{itemize}
    This completes the proof.
\end{proof}

For $n\geq 3$, we establish a  verification of \cref{conj. converse of prop} under a strengthened assumption.
\begin{theorem}\label{thm. converse for higher dimensional}
    Let $(M^n,g), n\geq 2$, be a complete noncompact manifold with nonnegative Ricci curvature. Assume $u\in C^3(M^n)$ is a solution to \cref{eq. n Liouville equation on M} satisfying
    \begin{align*}
        \limsup_{r(x)\to+\infty}
        \frac{u(x)}{\log r(x)}
        =-\infty,
    \end{align*}
    then
    \begin{align*}
        \frac{\int_{M^n}e^u}{\mathrm{AVR}(M^n,g)}=+\infty
    \end{align*}
    in the sense of \cref{two cases in conj}.
\end{theorem}
\begin{remark}
    Note that the example \cref{example vanishing AVR} indeed falls within the scope of \cref{thm. converse for higher dimensional}.
However, the complexity of the asymptotic behavior of solutions relevant to \cref{conj. converse of prop} goes far beyond the scope of this result.
    Indeed, the solution constructed in \cref{example mine} satisfies 
\begin{align*}
    \liminf_{r(x)\to+\infty}
    \frac{u(x)}{\log r(x)}=-\infty,\qquad
    \limsup_{r(x)\to+\infty} \frac{u(x)}{\log r(x)}=
    \begin{cases}
        +\infty &\text{if }n=2,\\
        0 &\text{if }n\geq 3.
    \end{cases}
\end{align*}

\end{remark}

The proof of \cref{thm. converse for higher dimensional} is based on the following general lemma:
\begin{lemma}\label{lem. lower bound of L1 norm}
    Let $(M^n,g), n\geq 2$, be a complete noncompact manifold with nonnegative Ricci curvature. Assume $u\in C^3(M^n)$ is a solution to
    \begin{align*}
        -\Delta_n u=f,\quad
        \text{on }M^n,
    \end{align*}
    for some function $f\in L^{\infty}_{\text{loc}}(M)$.
    Assume that
    \begin{align}\label{limsup assumption}
        \limsup_{r(x)\to+\infty}
        \frac{u(x)}{\log r(x)}
        \leq -\gamma<0
    \end{align}
    for some positive constant $\gamma$, then
    \begin{align*}
        \int_{M^n}f
        \geq |\mathbb{S}^{n-1}|
        \mathrm{AVR}(M^n,g)\gamma^{n-1},
    \end{align*}
    where $\mathrm{AVR}(M^n,g)$ is defined in \cref{def. AVR}.
\end{lemma}
\begin{proof}
    \cref{limsup assumption} yields that,
    for any fixed $\epsilon\in (0,\frac{\gamma}{2})$,  there exists $R_\epsilon>0$, such that
    \begin{align*}
        u(x)\leq (-\gamma+\epsilon)\log r(x),\quad 
        \forall r(x)>R_\epsilon.
    \end{align*}
    Consider the function
    \begin{align*}
        G(x):=-(\gamma-2\epsilon)
        \max\{\log r(x),0\},
    \end{align*}
    and the sets
    \begin{align*}
        U_t:=\{x\in M: G(x)-u(x)<t\},\quad \forall t>1.
    \end{align*}
    We first note that  each $U_t$ is a bounded domain, since for any $ r(x)>\max\{{\exp(\frac{t}{\epsilon})}, R_\epsilon, e\}$, there holds
    \begin{align*}
        G(x)
        =(-\gamma+2\epsilon)\log r(x)
        > (-\gamma+\epsilon)\log r(x)+t
        \geq u(x)+t.
    \end{align*}
    Therefore,
    \begin{align}\label{Ut is bounded}
        U_t\subset B(p,\rho_t),\quad 
        \text{where }\rho_t:=
        \max\left\{
        {\exp\left({\frac{t}{\epsilon}}\right)},R_\epsilon,\ e
        \right\}.
    \end{align}
    We also note that $\{U_t\}_{t>1}$ is an exhaustion of $M^n$. In fact, by the definition of $G$ and $U_t$, for any $R>0$, we have
    \begin{align*}
        B(p,R)\subset U_t,\qquad \forall 
        t>\sup_{B(p,R)}(-u).
    \end{align*}
    We claim that
    \begin{align}\label{ineq. comparison of normal derivatives}
        |\nabla u|^{n-2}\frac{\partial u}{\partial\nu}
        \leq |\nabla G|^{n-2}\frac{\partial G}{\partial\nu}\quad 
        \text{on }\partial U_t,
    \end{align}
    where $\nu$ is the unit outer normal vector on $\partial U_t$. First, by the definition of $U_t$, we know that
    \begin{align*}
        \frac{\partial u}{\partial \nu}
        \leq \frac{\partial G}{\partial\nu}
        \quad \text{on }\partial U_t.
    \end{align*}
    Now we shall verify \cref{ineq. comparison of normal derivatives} at any point $q\in\partial U_t$ in three cases separately.
    
    If $\frac{\partial u}{\partial\nu}\geq 0$ at $q$, noticing that $\nabla_{\partial U_t}u=\nabla_{\partial U_t}G$ at $q$, we get
    \begin{align*}
        |\nabla u|^2
        =|\nabla_{\partial U_t}u|^2
        +\left(\frac{\partial u}{\partial\nu}\right)^2
        \leq 
        |\nabla_{\partial U_t}G|^2
        +\left(\frac{\partial G}{\partial\nu}\right)^2
        =|\nabla G|^2.
    \end{align*}
    Hence \cref{ineq. comparison of normal derivatives} holds in this case.

    If $\frac{\partial u}{\partial\nu}<0\leq \frac{\partial G}{\partial\nu}$ at $q$, then \cref{ineq. comparison of normal derivatives} holds trivially.

    If $\frac{\partial u}{\partial\nu}
    \leq \frac{\partial G}{\partial\nu}<0$ at $q$, using $\nabla_{\partial U_t}u=\nabla_{\partial U_t}G$ at $q$, we get
    \begin{align*}
        |\nabla u|^2
        =|\nabla_{\partial U_t}u|^2
        +\left(\frac{\partial u}{\partial\nu}\right)^2
        \geq 
        |\nabla_{\partial U_t}G|^2
        +\left(\frac{\partial G}{\partial\nu}\right)^2
        =|\nabla G|^2.
    \end{align*}
    Then \cref{ineq. comparison of normal derivatives} still holds.

    Using   the divergence theorem, \cref{ineq. comparison of normal derivatives}, \cref{Ut is bounded} and the Laplace comparison theorem $\Delta_n\log r\leq 0$, we derive that
    \begin{align*}
        &\int_{U_t}f
        =-\int_{U_t}\Delta_n u
        =-\int_{\partial U_t}|\nabla u|^{n-2}\frac{\partial u}{\partial\nu}\\
        \geq& -\int_{\partial U_t}
        |\nabla G|^{n-2}\frac{\partial G}{\partial\nu}
        =\int_{\partial U_t}
        (\gamma-2\epsilon)^{n-1}
        |\nabla \log r|^{n-2}
        \frac{\partial\log r}{\partial\nu}\\
        =&(\gamma-2\epsilon)^{n-1}
        \int_{U_t\setminus B(p,\epsilon)}
        \Delta_n\log r
        +(\gamma-2\epsilon)^{n-1}
        \int_{\partial B(p,\epsilon)}
        |\nabla \log r|^{n-2}
        \frac{\partial\log r}{\partial\nu}\\
        \geq& (\gamma-2\epsilon)^{n-1}
        \int_{B(p,\rho_t)\setminus B(p,\epsilon)}
        \Delta_n\log r
        +(\gamma-2\epsilon)^{n-1}
        \int_{\partial B(p,\epsilon)}
        |\nabla \log r|^{n-2}
        \frac{\partial\log r}{\partial\nu}\\
        =&(\gamma-2\epsilon)^{n-1}
        \int_{\partial B(p,\rho_t)}
        |\nabla \log r|^{n-2}
        \frac{\partial\log r}{\partial\nu}\\
        =&(\gamma-2\epsilon)^{n-1}
        \frac{\mathrm{Area}(\partial B(p, \rho_t))}{\rho_t^{n-1}}.
    \end{align*}
    Taking $t$ to $+\infty$, it follows that
    \begin{align*}
        \int_M f
        \geq (\gamma-2\epsilon)^{n-1}
        |\mathbb{S}|^{n-1}
        \mathrm{AVR}(M^n,g).
    \end{align*}
    Then the proof finishes by taking $\epsilon$ to $0$.
\end{proof}
\begin{remark}
    Observe that we have not imposed any positivity assumption on the source function $f$, yet we have derived a positive lower bound for its integral over $M$.
    In this light, the logarithmic upper bound \cref{limsup assumption} also reveals an intriguing connection to the integral of the source term $f$.
\end{remark}
Furthermore, a similar argument to that in the proof of \cref{lem. lower bound of L1 norm} yields the following estimate, which likewise does not require nonnegativity of the source function and therefore improves upon \cref{lem. upper bound of L1 norm}.
\begin{lemma}\label{lem. upper bound of L1 norm_second}
    Let $(M^n,g), n\geq 2$, be a complete noncompact manifold with nonnegative Ricci curvature. Assume $u\in C^3(M^n)$ is a solution to
    \begin{align*}
        -\Delta_n u=f,\quad
        \text{on }M^n,
    \end{align*}
    for some function $f\in L^{\infty}_{\text{loc}}$. Assume 
    \begin{align*}
        \liminf_{r(x)\to+\infty}
        \frac{u(x)}{\log r(x)}\geq -\gamma,
    \end{align*}
    for some positive constant $\gamma>0$, then there holds
    \begin{align*}
        \int_M f
        \leq |\mathbb{S}^{n-1}|\mathrm{AVR}(M^n,g)\, \gamma^{n-1}<+\infty,
    \end{align*}
    where $\mathrm{AVR}(M^n,g)$ is defined in \cref{def. AVR}.
\end{lemma}

Now it's very easy to establish \cref{thm. converse for higher dimensional}.
\begin{proof}[Proof of \cref{thm. converse for higher dimensional}]
    By \cref{lem. lower bound of L1 norm} and our assumption, we know that
    \begin{align*}
        \int_{M^n}e^u
        \geq |\mathbb{S}^{n-1}|
        \mathrm{AVR}(M^n,g)\gamma^{n-1},\quad \forall \gamma>0.
    \end{align*}
    Letting $\gamma$ go to $+\infty$, we conclude \cref{two cases in conj}.
\end{proof}


\bibliographystyle{abbrv} 
\bibliography{references}
\end{document}